\documentclass[a4paper,reqno,10pt]{amsart}
\usepackage{amsfonts,amssymb,amsthm,amsbsy,latexsym,amscd,amsmath,euscript,enumerate,manfnt, marvosym,verbatim,calc,mathrsfs,marvosym,tensor,textcomp ,color,xcolor,tikz,setspace,upgreek,bbm,bm,datetime}

\definecolor{bulgarianrose}{rgb}{0.28, 0.02, 0.03}
\usepackage[linkcolor=bulgarianrose,citecolor=bulgarianrose,colorlinks=true,hypertexnames=false]{hyperref}

\newtheorem{theorem}{Theorem}[section]

\newtheorem{lemma}[theorem]{Lemma}

\theoremstyle{definition}

\newtheorem{construction}[theorem]{Construction}

\newtheorem*{acknowledgement}{Acknowledgement}

\newcommand{\tr}[2][]{\ensuremath{{\textnormal{tr}}_{#1}(#2)}}

\newcommand{\Pt}{\ensuremath{{\textnormal{P}}}}

\newcommand{\M}[2][]{\ensuremath{{\textnormal{M}}_{#1}(#2)}}

\makeatletter
\def\imod#1{\allowbreak\mkern10mu({\operator@font mod}\,\,#1)}
\def\@textbottom{\vskip\z@\@plus 18pt}
\let\@texttop\relax
\makeatother

\title[Transversal size of a series of families constructed over Cycle Graph]{\textnormal{A note on the transversal size of a series of families\\ constructed over Cycle Graph}}
\author{Kaushik Majumder}
\author{Satyaki Mukherjee}
\address{\newline Theoretical Statistics and Mathematics Unit\\ \newline Indian Statistical Institute, Bangalore Centre \\ \newline $8^{th}$ Mile Mysore Road, Bangalore - $560059$, India. \newline \textnormal{\textestimated-Mails}: {\tt kaushikbnmajumder\MVAt gmail.com} (\textnormal{Kaushik Majumder})\\\newline\hspace*{1.25cm}{\tt paglasatyaki\MVAt gmail.com} (\textnormal{Satyaki Mukherjee}).}

\date{January~9,2015}

\begin{document}

\begin{abstract}
Paul Erd\H{o}s and L\'{a}szl\'{o} Lov\'{a}sz established by means of an example that there exists a maximal intersecting family of $k-$sets with approximately $(e-1)k!$ blocks. L\'{a}szl\'{o} Lov\'{a}sz conjectured that their example is best known example which has the maximum number of blocks. Later it was disproved. But the quest for such examples remain valid till this date. In this short note, by computing transversal size of a certain series of uniform intersecting families constructed over the cycle graph, we provide an example which has more than $(\frac{k}{2})^{k-1}$ (approximately) blocks.   
\end{abstract}

\maketitle

\section{Introduction}

By a \emph{family} we mean a collection (set) of finite sets. A family is called \emph{intersecting} if any two of its members have non empty intersection. Given a family $\mathcal{G}$, the members of $\mathcal{G}$ are called its \emph{blocks} and the elements of the blocks are called its \emph{points}. The \emph{point set} of the family $\mathcal{G}$ is defined as $\underset{B\in\mathcal{G}}{\cup} B$ and is denoted by $\Pt_{\mathcal{G}}$. A family $\mathcal{G}$ is said to be \emph{uniform} if all its blocks have the same size. A uniform family with common block size $k$ is referred to as \emph{family of $k-$sets}. A set $C$ is said to be a \emph{blocking set} of a finite non empty family $\mathcal{G}$ if $C$ intersects every block of $\mathcal{G}$. A minimum size blocking set of $\mathcal{G}$ is called a \emph{transversal} of $\mathcal{G}$. We denote the common size of its transversals by $\tr{\mathcal{G}}$ and the family of transversals of $\mathcal{G}$ by $\mathcal{G}^{\top}$. Note that $\mathcal{G}^{\top}
$ is a uniform family.

Let $k$ and $t$ be positive integers. A family of $k-$sets $\mathcal{F}$ is said to be a \emph{maximal intersecting family of $k-$sets} if $\mathcal{F}=\mathcal{F}^{\top}$. It is not clear from the definition whether such a family has finite number of blocks. Erd\H{o}s and Lov\'{a}sz proved the surprising result that any such family has at most $k^{k}$ blocks (see \cite[Theorem~7]{MR0382050}). This result is of central attraction in the study of intersecting family of $k-$sets with finite transversal size. It allows us to define the integer $\M{k}$ to be the maximum number of blocks achievable by any maximal intersecting family of $k-$sets. So it is very natural to ask for the question to find a maximal intersecting family of $k-$sets with $\M{k}$ blocks. In \cite{aX14111480}, it is answered by means of an example that $\M{k}\geq(\frac{k}{2})^{k-1}$. The core part of this example is to produce an intersecting family of $k-$sets with transversal size $t\leq k-1$ in such a way that it can be embedded in a 
maximal intersecting family of $k-$sets. To fulfil this purpose it is very relevant to study this core part. This leads to the concept of a ``closure property'', which is studied in \cite[\textsection~2]{aX14111480}. 

In this note we construct (see $\mathbb{F}(k,t)$ in Construction~\ref{F(k,t)} below) a series of intersecting families of $k-$sets with transversal size $t\leq k-1$ such that each such family can be embedded in a maximal intersecting family of $k-$sets. Such a construction is not entirely new. There are similar type of families, namely $\mathscr{G}$ in \cite[\textsection~2]{MR1383503}. However, the compact description given here is amenable to rigorous arguments. Our purpose of this note is to show that transversal size of $\mathbb{F}(k,t)$ is $t$ (Theorem~\ref{trF(k,t)}) and as a consequence we have, for any positive integer $k$,
\begin{equation}\label{M(k)}\tag{$\star$}
\M{k}\geq|\mathbb{F}(k,k-1)|+|\mathbb{F}^{\top}(k,k-1)|>\left(\frac{k}{2}\right)^{k-1}.
\end{equation}

\begin{construction}\label{F(k,t)}
Let $k$ and $t$ be positive integers with $t\leq k$. Let $X_{n}$ $0\leq n\leq t-1$, be $t$ pairwise disjoint sets with  
\begin{align*}
|X_{n}|=\left\{\begin{array}{lcr}
                k-\lfloor\frac{t}{2}\rfloor & \textnormal{if} & 0\leq n\leq\lfloor\frac{t-1}{2}\rfloor\\
                k-\lfloor\frac{t-1}{2}\rfloor & \textnormal{if} & \lfloor\frac{t-1}{2}\rfloor+1\leq n\leq t-1
            \end{array}\right.
\end{align*}
say $X_{n}=\{x^{n}_{p}:0\leq p\leq |X_{n}|-1\}$. Let $\mathbb{F}(k,t)$ be the family of all the $k-$sets of the form 
\begin{equation*}
X_{n}\sqcup\left\{x^{n+i}_{p_{i}}:1\leq i\leq k-|X_{n}|\right\},
\end{equation*}
where $0\leq n\leq t-1$, addition in the superscript is modulo $t$ and $\{p_{m}:m\geq 0\}$ varies over all finite sequences of non negative integers satisfying,  
\begin{equation}\label{definition_p_n}\tag{$\star\star$}
p_{0}=0\textup{ and for }m\geq1, p_{m}=p_{m-1}\textup{ or } 1+p_{m-1}.   
\end{equation}
\end{construction}

In this construction, the pairwise disjoint sets $X_{n}$ may be thought as arranged along a $t-$cycle. Since the diameter of a $t-$cycle is $\lfloor\frac{t}{2}\rfloor$, it is easy to verify that $\mathbb{F}(k,t)$ is an intersecting family of $k-$sets.

\begin{theorem}\label{trF(k,t)}
$\tr{\mathbb{F}(k,t)}=t$.
\end{theorem}

By using Theorem~\ref{trF(k,t)}, we have $\{x_{i}\in X_{i}:0\leq i\leq t-1\}$ is a transversal of $\mathbb{F}(k,t)$.
Therefore, there are $\overset{t-1}{\underset{i=0}\prod}|X_{i}|$ choices for such transversals. But there are other transversals. Hence
\begin{align*}
|\mathbb{F}^{\top}(k,t)|>\left\{\begin{array}{lcl}
                (k-r+1)^{2r-1} & \textnormal{if} & t=2r-1\\
                (k-r)^{r}(k-r+1)^{r} & \textnormal{if} & t=2r.
            \end{array}\right.
\end{align*}

Let $\mathcal{A}$ be a maximal intersecting family of $(k-t)-$sets. Let $\Pt_{\mathcal{A}}$ and $\Pt_{\mathbb{F}^{\top}(k.t)}$ be disjoint. By \cite[Proposition~3.4]{aX14111480} and \cite[Theorem~2.7]{aX14111480}, it follows that $\mathbb{F}(k,t)\sqcup\{\mathcal{A}\circledast\mathbb{F}^{\top}(k,t)\}$ is a maximal intersecting family of $k-$sets. Here $\mathcal{A}\circledast\mathbb{F}^{\top}(k,t)$ denotes the collection of all sets of the form $A\sqcup T$, where $A\in\mathcal{A}$ and $T\in\mathbb{F}^{\top}(k,t)$. If we consider the case $t=k-1$, we have  $\mathbb{F}(k,k-1)\sqcup\{\mathcal{A}\circledast\mathbb{F}^{\top}(k,k-1)\}$ is a maximal intersecting family of $k-$sets and as a consequence we deduce \eqref{M(k)}.

\section{Proof of Theorem~\ref{trF(k,t)}}

The following remarkable lemma is essentially the case $n=1$ of \cite[Theorem~2.1]{MR0114765}. Since the original proof is obscured by many hypotheses and technical terms, we include a simpler proof for the sake of completeness.

Recall that, for any finite sequence $(r_{1},\ldots,r_{t})$ its cyclic shifts are the $t$ sequences $(r_{i+1},\ldots,r_{i+t})$ where $0\leq i\leq t-1$ and the addition in the subscripts is modulo $t$.

\begin{lemma}[Raney]\label{Raney_lemma}
Let $(r_{1},r_{2},\ldots,r_{t})$ be a finite sequence of integers such that $\overset{t}{\underset{i=1}\sum}r_{i}=1$. Then, exactly one of the of the $t$ cyclic shifts of this sequence has all its partial sums strictly positive.
\end{lemma}
\begin{proof}
For $1\leq n\leq t$, let $s_{n}=r_{1}+\ldots+r_{n}-\frac{n}{t}$. Suppose, if possible, $s_{m}=s_{n}$ for some indexes $1\leq m<n\leq t$. Then $r_{m+1}+\ldots+r_{n}=\frac{n-m}{t}$, which is a contradiction, since the left hand side is an integer and the right hand side is a proper fraction. Thus, the $t$ numbers $s_{i}$ are distinct. So there is a unique index $\mu$, with $1\leq\mu\leq t$, for which $s_{\mu}$ is the minimum of these $t$ numbers. Now, for $\mu+1\leq m\leq t$,
\begin{equation*}
r_{\mu+1}+\ldots+r_{m}=(s_{m}-s_{\mu})+\frac{m-\mu}{t}>0
\end{equation*}
and for $1\leq m\leq\mu$,
\begin{align*}
r_{\mu+1}+\ldots+r_{t}+r_{1}+\ldots+r_{m}&=1-(s_{\mu}+\frac{\mu}{t})+(s_{m}+\frac{m}{t})\\
&=(s_{m}-s_{\mu})+1-\frac{\mu-m}{t}>0.
\end{align*}
Thus, the partial sums of $(r_{\mu+1},\ldots,r_{\mu+t})$ are all strictly positive. This proves the existence.

Conversely, let $\mu$ be an index for which the partial sums of $(r_{\mu+1},\ldots,r_{\mu+t})$ are all strictly positive. Then each of these partial sums is at least $1$, so that if we subtract a proper fraction from one of them, then the result remains positive. For $\mu+1\leq m\leq t$, 
\begin{equation*}
s_{m}-s_{\mu}=(r_{\mu+1}+\ldots+r_{m})-\frac{\mu-m}{t}\geq0
\end{equation*}
and for $1\leq m\leq\mu$,
\begin{equation*}
s_{m}-s_{\mu}=(r_{\mu+1}+\ldots+r_{t}+r_{1}+\ldots+r_{m})-\frac{\mu-m}{t}\geq0
\end{equation*}
Thus $\mu$ is the unique index for which $s_{\mu}=\min\{s_{i}:1\leq i\leq t\}$. This proves the uniqueness.
\end{proof}

\begin{proof}[\tt{Proof of Theorem~\ref{trF(k,t)}} :]
If $C$ is any $t-$set which intersects each $X_{n}$ in a singleton, then in particular $C$ is a blocking set of $\mathbb{F}(k,t)$. So $\tr{\mathbb{F}(k,t)}\leq t$. So, it suffices to show that $\mathbb{F}(k,t)$ has no blocking set $C$ of size $t-1$. Assume the contrary. For $0\leq n\leq t-1$,  $|C\cap X_{n}|$ is a non negative integer and $\overset{t-1}{\underset{i=0}\sum}|C\cap X_{i}|=t-1$. Therefore, if we define the integers $r_{n+1}=1-|C\cap X_{n}|$, where $0\leq n\leq t-1$, then $\overset{t}{\underset{i=1}\sum}r_{i}=1$. So applying Lemma~\ref{Raney_lemma} to this sequence, we get a unique $0\leq \mu\leq t-1$ such that $\overset{n}{\underset{i=0}\sum}r_{\mu+i}\geq1$, i.e. 
$|C\cap(\overset{n}{\underset{i=0}\sqcup}X_{\mu+i})|\leq n$, for $0\leq n\leq t-1$. In particular, $C$ is disjoint from $X_{\mu}$. For $1\leq n\leq k-|X_{\mu}|$, let $l_{n}=n-\overset{n}{\underset{i=1}\sum}|C\cap X_{\mu+i}|$. Thus $l_{n}\geq0$. Let $P_{n}$ be the set of all integers $p\geq 0$ for which 
there is a sequence $(p_{1},\ldots,p_{n})$ satisfying \eqref{definition_p_n} such that $p_{n}=p$ and for $1\leq i\leq n$, $x^{\mu+i}_{p_{i}}\notin C$. 

\noindent{\textsl{Claim} :}  $|P_{n}|\geq1+l_{n}$ for $1\leq n\leq k-|X_{\mu}|$.

\begin{proof}[\tt{Proof of Claim} :]\renewcommand{\qedsymbol}{}
We prove it by finite induction on $n$. When $n=1$, 
\begin{align*}
|P_{n}|&=2-|C\cap X_{\mu+n}|\\
&=1+l_{n}.
\end{align*}
So the claim is true for $n=1$.

Now let $1\leq m\leq k-1-|X_{\mu}|$ and suppose that the claim is true for $m$. Since $|C\cap X_{\mu+m+1}|=1+l_{m}-l_{m+1}$ and clearly 
\begin{equation*}
P_{m+1}\supseteqq(P_{m}\cup\{1+p:p\in P_{m}\})\smallsetminus(C\cap X_{\mu+m+1}),
\end{equation*}
we have 
\begin{align*}
|P_{m+1}|&\geq|P_{m}\cup\{1+p:p\in P_{n}\}|-|C\cap X_{\mu+m+1}|\\
&\geq 1+|P_{m}|-|C\cap X_{\mu+m+1}|\\
&\geq 2+l_{m}-(1+l_{m}-l_{m+1})\\
&=1+l_{m+1}
\end{align*}
This completes the induction and proves the claim.
\end{proof}
By the case $n=k-|X_{\mu}|$ of the claim, $P_{k-|X_{\mu}|}$ is non empty. Hence there is a sequence 
$\{p_{1},\ldots,p_{k-|X_{\mu}|}\}$ satisfying \eqref{definition_p_n} and disjoint from $C$. Therefore, the block 
$X_{\mu}\sqcup\{p_{i}:1\leq k-|X_{\mu}|\}$ is disjoint from $C$. Thus $C$ is not a blocking set of $\mathbb{F}(k,t)$. Since $C$ is an arbitrary set of size $t-1$, this shows $\tr{\mathbb{F}(k,t)}\geq t$.
\end{proof}

\begin{acknowledgement}
The authors would like to thank Professor Bhaskar Bagchi for pointing out the reference \cite{MR0114765}, which was unknown to us, and for his help in the preparation of this note. It improves the presentation of this note significantly. 
\end{acknowledgement}

\end{document}